\documentclass[12pt]{article}%
\usepackage{amsmath}
\usepackage{amsfonts}
\usepackage{amssymb}
\usepackage{graphicx}%
 \makeatletter
\newfont{\footsc}{cmcsc10 at 8truept}
\newfont{\footbf}{cmbx10 at 8truept}
\newfont{\footrm}{cmr10 at 10truept}
\newtheorem{theorem}{Theorem}

\newtheorem{proposition}[theorem]{Proposition}

\newenvironment{proof}[1][Proof]{\noindent{\textbf {#1}  }}  {\hfill$\Box$\bigskip}

\def\blfootnote{\xdef\@thefnmark{}\@footnotetext}
\begin{document}

\title{On the second largest eigenvalue of the signless Laplacian}
\author{Leonardo Silva de Lima\thanks{Department of Production Engineering, Federal
Center of Technological Education, Rio de Janeiro, Brazil;\textit{email:
llima@cefet-rj.br, llima@memphis.edu}} \thanks{Research supported by CNPq
Grant 201888/2010--6.} \ and$\;$Vladimir Nikiforov\thanks{Department of
Mathematical Sciences, University of Memphis, Memphis TN 38152, USA;
\textit{email: vnikifrv@memphis.edu}} \thanks{Research supported by NSF Grant
DMS-0906634.} }
\date{}
\maketitle

\begin{abstract}
Let $G$ be a graph of order $n,$ and let $q_{1}\left(  G\right)  \geq
\cdots\geq q_{n}\left(  G\right)  $ be the eigenvalues of the $Q$-matrix of
$G$, also known as the signless Laplacian of $G.$ We give a necessary and
sufficient condition for the equality $q_{k}\left(  G\right)  =n-2,$ where
$1<k\leq n.$ In particular, this result solves an open problem raised by Wang,
Belardo, Huang and Borovicanin.

We also show that
\[
q_{2}\left(  G\right)  \geq\delta\left(  G\right)
\]
and determine all graphs for which equality holds. \medskip



\end{abstract}

\section{Introduction and main results}

Given a graph $G,$ write $A$ for the adjacency matrix of $G$ and let $D$ be
the diagonal matrix of the row-sums of $A,$ i.e., the degrees of $G.$ The
matrix $Q\left(  G\right)  =A+D,$ called the \emph{signless Laplacian} or the
$Q$-matrix of $G,$ has been intensively studied recently; see, e.g.,
Cvetkovi\'{c} \cite{C10} for a comprehensive survey.

As usual, we shall index the eigenvalues of $Q\left(  G\right)  $ in
non-increasing order and denote them as $q_{1}\left(  G\right)  ,q_{2}\left(
G\right)  ,\ldots,q_{n}\left(  G\right)  .$ Also, we shall write $\overline
{G}$ for the complement of $G.$

This paper is about the second largest $Q$-eigenvalue of a graph. Some notable
contributions to this area have been made by Yan \cite{Ya02}, Cvetkovi\'{c},
Rowlinson and Simi\'{c} \cite{CRS07}, Wang \emph{et al.} \cite{WBHB10}, Das
\cite{Das10}, \cite{Das11}, and Aouchiche, Hansen and Lucas \cite{AHL11}.

First, in \cite{Ya02}, Yan has proved that if $G$ is a graph of order
$n\geq2,$ then $q_{2}\left(  G\right)  \leq n-2.$ It is easy to see that
equality holds when $G$ is the complete graph, but there are many other graphs
with this property, see, e.g., \cite{AHL11} and \cite{WBHB10} for particular
examples. Rather naturally, the authors of \cite{WBHB10} raised the problem to
characterize all graphs $G$ of order $n\geq2$ such that
\begin{equation}
q_{2}\left(  G\right)  =n-2. \label{mi}%
\end{equation}
The first result of this paper gives a complete solution to this problem, but
in fact our methods allow to answer a more general question. To state these
results, we need the following definition.\medskip

\textbf{Definition }\emph{A connected bipartite graph is called
\textbf{balanced} if the sizes of its vertex classes are equal, and
\textbf{unbalanced} otherwise. An isolated vertex is considered to be an
unbalanced bipartite graph with an empty vertex class.\medskip}

\begin{theorem}
\label{th1} If $G$ is a graph of order $n\geq2$, then $q_{2}\left(  G\right)
=n-2$ if and only if $\overline{G}$ has a balanced bipartite component or at
least two bipartite components.
\end{theorem}

Note that if $G$ is a graph of order $n$ and $1<k\leq n,$ then $q_{k}\left(
G\right)  \leq n-2$, and since equality is attained for the complete graph it
is natural to ask a more general question: \medskip

\emph{Given }$k\geq2,$\emph{ for which graphs }$G$\emph{ of order }$n$\emph{
it is true that }$q_{k}\left(  G\right)  =n-2$ $?$ \medskip

In other words, how the structure of a graph of order $n$ relates to the
multiplicity of the $Q$-eigenvalue $n-2$ ? Our approach allows to specify this
relation precisely.

\begin{theorem}
\label{th2} Let $1\leq k<n$ and let $G$ be a graph of order $n.$ Then
$q_{k+1}\left(  G\right)  =n-2$ if and only if $\overline{G}$ has either $k$
balanced bipartite components or $k+1$ bipartite components.
\end{theorem}

We conclude the paper by comparing $q_{2}\left(  G\right)  $ to the minimum
degree $\delta\left(  G\right)  $. In \cite{Das10}, Das proved that
\[
q_{2}(G)\geq\overline{d}\left(  G\right)  -1\text{ and }q_{2}(G)\geq\Delta
_{2}(G)-1,
\]
where $\overline{d}\left(  G\right)  $ and $\Delta_{2}(G)$ are the average and
the second largest degrees of $G,$ respectively. In the light of these
inequalities the bound given below looks easy, but its proof is not immediate
and it is sharp for many different kinds of graphs. To state the result, write
$K_{n_{1},n_{2},\ldots,n_{r}}$ for the complete $r$-partite graph with class
sizes $n_{1},\ldots,n_{r}.$

\begin{theorem}
\label{th3} If $G$ is a noncomplete graph of order $n$, then
\[
q_{2}\left(  G\right)  \geq\delta\left(  G\right)  .
\]
Equality holds if and only if $G$ is one of the following graphs: a star, a
complete regular multipartite graph, the graph $K_{1,3,3},$ or a complete
multipartite graph of the type $K_{1,\ldots,1,2,\ldots,2}$.
\end{theorem}

\section{Notation and preliminary results}

In general, our notation follows \cite{B98}; thus, for a graph $G$ we write
$V\left(  G\right)  $ for the vertex set of $G$ and $E\left(  G\right)  $ for
the edge set of $G.$ We write $\overline{G}$ stands for the complement of $G,$
and $K_{n}$ for the complete graph on $n$ vertices.

Furthermore, given a Hermitian matrix $A$ of order $n,$ we index its
eigenvalues as $\lambda_{1}(A)\geq\cdots\geq\lambda_{n}(A).$ We retain the
same notation for the eigenvalues of the self-adjoint operator defined by the
matrix $A.$ The inner product of two vectors $\mathbf{x}$ and $\mathbf{y}$ is
denoted by $\left\langle \mathbf{x},\mathbf{y}\right\rangle $ and
$\mathbf{j}_{n}$ stands for the $n$-vector of all ones.\medskip

Recall that Desao and Rao in \cite{DR94}, Proposition 2, proved the important
fact that $0$ is an eigenvalue of $Q(G)$ if and only if $G$ has a bipartite
component. In fact, as shown in \cite{CRS07b}, Corollary 2.2, the following
precise statement holds.

\begin{theorem}
[\cite{CRS07b}, \cite{DR94}]\label{th4} Given a graph $G,$ the multiplicity of
$0$ as eigenvalue of $Q\left(  G\right)  $ is equal to the number of bipartite
components of $G.$
\end{theorem}

In our proofs, we shall often use the following basic fact about $Q(G)$:
\emph{if }$G$ \emph{is a }$\emph{graph}$\emph{ order }$n$\emph{ and
}$\mathbf{x}=\left(  x_{1},\ldots,x_{n}\right)  $ \emph{is an }$n$%
\emph{-vector}$,$\emph{ then}
\begin{equation}
\left\langle Q(G)\mathbf{x},\mathbf{x}\right\rangle =\sum_{uv\in E\left(
G\right)  }\left(  x_{u}+x_{v}\right)  ^{2}. \label{Qf}%
\end{equation}

Also, in the proofs of Theorem \ref{th1} and Theorem \ref{th2}, we shall use
Weyl's inequalities for eigenvalues of Hermitian matrices. Although these
fundamental inequalities have been known for almost a century, it seems that
their equality case was first established only recently, by So in \cite{So94},
and his work was inspired by the paper of Ikebe, Inagaki and Miyamoto
\cite{IIM87}.

For convenience we state below the complete theorem of Weyl and So.\medskip

\begin{theorem}
[\cite{So94}]\label{th:weyl} Let $A$ and $B$ be Hermitian matrices of order
$n,$ and let $1\leq i\leq n$ and $1\leq j\leq n.$ Then
\begin{align}
\lambda_{i}(A)+\lambda_{j}(B)  &  \leq\lambda_{i+j-n}(A+B),\text{ if }i+j\geq
n+1,\label{Wein1}\\
\lambda_{i}(A)+\lambda_{j}(B)  &  \geq\lambda_{i+j-1}(A+B),\text{ if }i+j\leq
n+1. \label{Wein2}%
\end{align}
In either of these inequalities equality holds if and only if there exists a
nonzero $n$-vector that is an eigenvector to each of the three involved eigenvalues.
\end{theorem}

Theorem \ref{th:weyl} is crucial for our proof of Theorem \ref{th1}, but it is
not general enough for the more complicated Theorem \ref{th2}. Therefore, we
need a strengthening of Theorem \ref{th:weyl} in a particular situation. To
begin with, note that Theorem \ref{th:weyl} can be stated equivalently if we
replace \textquotedblleft Hermitian matrices\textquotedblright\ by
\textquotedblleft self-adjoint linear operators\textquotedblright; indeed, the
latter setup seems even more natural.

\begin{proposition}
\label{pro1} Let $2\leq k<n$ and $A$ and $B$ be self-adjoint operators of
order $n.$ If for every $s=2,\ldots,k,$
\begin{equation}
\lambda_{s}\left(  A\right)  +\lambda_{n}\left(  B\right)  =\lambda_{s}\left(
A+B\right)  , \label{con1}%
\end{equation}
then there exist $k-1$ nonzero orthogonal $n$-vectors $\mathbf{x}^{1}%
,\ldots,\mathbf{x}^{k-1}$ such that%
\begin{equation}
A\mathbf{x}^{s-1}=\lambda_{s}\left(  A\right)  \mathbf{x}^{s-1},\text{\ \ }%
B\mathbf{x}^{s-1}=\lambda_{n}\left(  B\right)  \mathbf{x}^{s-1},\text{\ \ and
\ }\left(  A+B\right)  \mathbf{x}^{s-1}=\lambda_{s}\left(  A+B\right)
\mathbf{x}^{s-1} \label{req}%
\end{equation}
for every $s=2,\ldots,k.$
\end{proposition}

\begin{proof}
Our proof is by induction on $k.$ For $k=2$ the assertion follows from Theorem
\ref{th:weyl} since we have to find a single vector satisfying the
requirements (\ref{req}). Assume now that $k>2$ and that the assertion holds
for $2\leq k^{\prime}<k.$ By Theorem \ref{th:weyl}, there exists a nonzero
vector $\mathbf{x=x}^{k-1}$ such that%
\[
A\mathbf{x}^{k-1}=\lambda_{2}\left(  A\right)  \mathbf{x}^{k-1},\text{\ \ }%
B\mathbf{x}^{k-1}=\lambda_{n}\left(  B\right)  \mathbf{x}^{k-1}\text{ \ and
\ }\left(  A+B\right)  \mathbf{x}^{k-1}=\lambda_{2}\left(  A+B\right)
\mathbf{x}^{k-1}.
\]
Write $H$ for the orthogonal complement of $\mathbf{x}^{k-1}.$ Since $A$ and
$B$ are self-adjoint, $H$ is an invariant subspace of the three operators $A,$
$B$ and $A+B.$ Set $A^{\prime}=A|H$ and $B^{\prime}=B|H;$ then clearly
$A^{\prime}+B^{\prime}=\left(  A+B\right)  |H.$ Note that
\begin{align*}
\lambda_{1}\left(  A^{\prime}\right)   &  =\lambda_{1}\left(  A\right)
,\text{ \ }\lambda_{2}\left(  A^{\prime}\right)  =\lambda_{3}\left(  A\right)
,\text{ \ }\ldots\text{ \ },\lambda_{n-1}\left(  A^{\prime}\right)
=\lambda_{n}\left(  A\right)  ,\text{\ }\\
\lambda_{1}\left(  B^{\prime}\right)   &  =\lambda_{1}\left(  B\right)
,\text{ \ }\lambda_{2}\left(  B^{\prime}\right)  =\lambda_{2}\left(  B\right)
,\text{ \ }\ldots\text{\ \ },\lambda_{n-1}\left(  B^{\prime}\right)
=\lambda_{n-1}\left(  B\right)  ,\\
\lambda_{1}\left(  A^{\prime}+B^{\prime}\right)   &  =\lambda_{1}\left(
A+B\right)  ,\text{ \ }\lambda_{2}\left(  A^{\prime}+B^{\prime}\right)
=\lambda_{3}\left(  A+B\right)  ,\text{\ \ }\ldots\text{ \ },\lambda
_{n-1}\left(  A^{\prime}+B^{\prime}\right)  =\lambda_{n}\left(  A+B\right)  .
\end{align*}
Hence, equalities (\ref{con1}) imply that
\[
\lambda_{s}A^{\prime}+\lambda_{n-1}B^{\prime}=\lambda_{s}\left(  A^{\prime
}+B^{\prime}\right)  ,
\]
for $s=2,\ldots,k-1.$ Therefore, there exist $k-2$ nonzero orthogonal vectors
$\mathbf{y}^{1},\ldots,\mathbf{y}^{k-2}$ in $H$ such that%
\[
A^{\prime}\mathbf{y}^{s-1}=\lambda_{s}(A^{\prime})\mathbf{y}^{s-1},\text{
\ \ }B^{\prime}\mathbf{y}^{s-1}=\lambda_{n-1}(B^{\prime})\mathbf{y}%
^{s-1},\text{\ \ \ and \ \ }\left(  A^{\prime}+B^{\prime}\right)
\mathbf{y}^{s-1}=\lambda_{s}\left(  A^{\prime}+B^{\prime}\right)
\mathbf{y}^{s-1}%
\]
Considering $H$ as a subspace of $\mathbb{C}^{n},$ the vectors $\mathbf{y}%
^{1},\ldots,\mathbf{y}^{k-2}$ correspond to $n$-vectors $\mathbf{x}^{1}%
,\ldots,\mathbf{x}^{k-2},$ which together with $\mathbf{x}^{k-1}$ have the
desired properties. This completes the induction step and the proof of the proposition.
\end{proof}

We note that the above proposition is tailored to our needs; clearly other
generalizations in the same vein are possible.\bigskip

\section{Proofs of Theorems \ref{th1}, \ref{th2} and \ref{th3}}

\begin{proof}
[\textbf{Proof of Theorem \ref{th1}}]Assume first that $q_{2}(G)=n-2.$
Applying Weyl's inequality (\ref{Wein1}), we find that
\begin{equation}
q_{2}\left(  G\right)  +q_{n}\left(  \overline{G}\right)  \leq q_{2}\left(
K_{n}\right)  . \label{Win}%
\end{equation}
Since $q_{2}\left(  K_{n}\right)  =n-2,$ we see that $q_{n}\left(
\overline{G}\right)  =0$ and Theorem \ref{th4} implies that $\overline{G}$ has
a bipartite component. Also, since equality holds in (\ref{Win}), by Theorem
\ref{th:weyl}, there exists a unit vector $\mathbf{x}=\left(  x_{1}%
,\ldots,x_{n}\right)  $ that is an eigenvector to each of the eigenvalues
$q_{2}\left(  G\right)  ,q_{n}\left(  \overline{G}\right)  $ and $q_{2}\left(
K_{n}\right)  .$ The latter implies that $\sum_{i=1}^{n}x_{i}=0$ as
$\mathbf{x}$ is orthogonal to the eigenspace of $q_{1}\left(  K_{n}\right)  $,
which is $Span\left(  \mathbf{j}_{n}\right)  $.

Using (\ref{Qf}), we see that
\begin{equation}
0=q_{n}(\overline{G})=\left\langle Q(\overline{G})\mathbf{x},\mathbf{x}%
\right\rangle =\sum_{ij\in E\left(  \overline{G}\right)  }\left(  x_{i}%
+x_{j}\right)  ^{2}. \label{Gb}%
\end{equation}
Therefore, if $\overline{G}$ has just one bipartite component, say $H$, we
have $x_{w}=0$ for all vertices $w\in V\left(  G\right)  \backslash V\left(
H\right)  $ and $x_{u}=-x_{v}$ for each edge $uv\in E\left(  H\right)  .$ This
means that for every $u\in V\left(  H\right)  $ the entry $x_{u}$ takes one of
two possible values, which have opposite signs.\ Now the condition $\sum
_{i=1}^{n}x_{i}=0$ implies that the vertex classes of $H$ are equal in size
and so $H$ is balanced. This completes the proof of the \textquotedblleft only
if\textquotedblright\ part of the theorem.

Assume now that $\overline{G}$ has a balanced bipartite component, and let $U$
and $W$ be its vertex classes. Define a vector $\mathbf{x}=\left(
x_{1},\ldots,x_{n}\right)  $ as
\[
x_{u}=\left\{
\begin{array}
[c]{rl}%
1, & \text{if }u\in U;\\
-1, & \text{if }u\in W;\\
0, & \text{if }u\in V\left(  G\right)  \backslash\left(  U\cup W\right)  .
\end{array}
\right.
\]
Since $\left\vert U\right\vert =\left\vert W\right\vert ,$ we see that
$q_{n}\left(  \overline{G}\right)  \left\Vert \mathbf{x}\right\Vert
^{2}=\left\langle Q\left(  \overline{G}\right)  \mathbf{x},\mathbf{x}%
\right\rangle =0$ and so $\mathbf{x}$ is an eigenvector to $q_{n}\left(
\overline{G}\right)  .$ Also, $\sum_{i=1}^{n}x_{i}=0$ and so $\mathbf{x}$ is
orthogonal to $Span\left(  \mathbf{j}_{n}\right)  ;$ therefore $\mathbf{x}$ is
an eigenvector to $q_{2}\left(  K_{n}\right)  .$ Hence $n-2$ is an eigenvalue
to $Q(G)$ with eigenvector $\mathbf{x}$. If $G$ is connected, the
Perron-Frobenius theorem implies that $\mathbf{x}$ is not an eigenvector to
$q_{1}\left(  G\right)  $ because it has negative entries; therefore,
$q_{2}\left(  G\right)  =n-2.$ If $G$ is not connected, then $\overline{G}$ is
a connected balanced bipartite graph and so $G=2K_{n/2};$ therefore
$q_{2}\left(  G\right)  =n-2.$

Let now $\overline{G}$ have two bipartite components, which can be assumed
unbalanced as otherwise the proof is completed by the previous argument.
Denote the vertex classes of the one component by $U,W$ and the parts of the
other by $X,Y,$ and define a vector $\mathbf{x}=\left(  x_{1},\ldots
,x_{n}\right)  $ as
\[
x_{u}=\left\{
\begin{array}
[c]{rl}%
1, & \text{if }u\in U;\\
-1, & \text{if }u\in W;\\
\frac{\left\vert W\right\vert -\left\vert U\right\vert }{|X|-|Y|}, & \text{if
}u\in X;\\
\frac{\left\vert U\right\vert -\left\vert W\right\vert }{|X|-|Y|}, & \text{if
}u\in Y;\\
0, & \text{if }u\in V\left(  G\right)  \backslash\left(  U\cup W\cup U\cup
W\right)  .
\end{array}
\right.
\]
Since for each edge $uv\in E\left(  \overline{G}\right)  ,$ we have
$x_{u}=-x_{v},$ it turns out $q_{n}\left(  \overline{G}\right)  \left\Vert
\mathbf{x}\right\Vert ^{2}=\left\langle Q\left(  \overline{G}\right)
\mathbf{x},\mathbf{x}\right\rangle =0$ and $\mathbf{x}$ is an eigenvector to
$q_{n}\left(  \overline{G}\right)  .$ Also, we find that%
\[
\sum_{i=1}^{n}x_{i}=\left\vert U\right\vert -\left\vert W\right\vert
+\left\vert X\right\vert \frac{\left\vert W\right\vert -\left\vert
U\right\vert }{|X|-|Y|}+\left\vert Y\right\vert \frac{\left\vert U\right\vert
-\left\vert W\right\vert }{|X|-|Y|}=0,
\]
and so $\mathbf{x}$ is orthogonal to $Span\left(  \mathbf{j}_{n}\right)  $ and
therefore $\mathbf{x}$ is an eigenvector to $q_{2}(K_{n}).$ Hence $n-2$ is an
eigenvalue to $Q(G)$ with eigenvector $\mathbf{x.}$ Since $\overline{G}$ has
at least two components, $G$ is connected and the Perron-Frobenius theorem
implies that and $\mathbf{x}$ is not an eigenvector to $q_{1}\left(  G\right)
$ because it has negative entries; therefore $q_{2}\left(  G\right)  =n-2.$
This completes the proof of the theorem.
\end{proof}

\begin{proof}
[\textbf{Proof of Theorem \ref{th2}}]Note that Theorem \ref{th1} covers the
case $k=1,$ so we shall assume that $k>1.$ For convenience we split the
theorem into two statements:

(A) If $q_{k+1}\left(  G\right)  =n-2,$ then $\overline{G}$ has either $k$
balanced bipartite components or at least $k+1$ bipartite components;

(B) If $\overline{G}$ has either $k$ balanced bipartite components or $k+1$
bipartite components, then $q_{k+1}\left(  G\right)  =n-2.$

First we prove (A). If $q_{k+1}\left(  G\right)  =n-2,$ then obviously
\[
q_{2}\left(  G\right)  =\cdots=q_{k+1}\left(  G\right)  =n-2,
\]
and so, for $i=2,\ldots,k+1,$ Weyl's inequalities (\ref{Wein1}) and
(\ref{Wein2}) imply that
\begin{equation}
n-2\leq q_{i}\left(  G\right)  +q_{n}\left(  \overline{G}\right)  \leq
q_{i}\left(  K_{n}\right)  =n-2. \label{Win1}%
\end{equation}
We see that for every $i=2,\ldots,k+1,$ equality holds throughout
(\ref{Win1}). In view of this fact, Proposition \ref{pro1} implies that there
exist $k$ nonzero orthogonal $n$-vectors $\mathbf{x}^{1},\ldots,\mathbf{x}%
^{k}$ such that for $s=1,\ldots,k,$
\[
Q(G)\mathbf{x}^{s}=q_{s+1}\mathbf{x}^{s},\text{ \ }Q(\overline{G}%
)\mathbf{x}^{s}=q_{n}\left(  \overline{G}\right)  \mathbf{x}^{s},\text{\ \ and
\ }Q\left(  K_{n}\right)  \mathbf{x}^{s}=q_{s+1}\left(  K_{n}\right)
\mathbf{x}^{s}.
\]
These equalities in particular imply that
\[
q_{n}\left(  \overline{G}\right)  =\cdots=q_{n+1-k}\left(  \overline
{G}\right)  =0.
\]
Hence, by Theorem \ref{th4}, $\overline{G}$ has at least $k$ bipartite
components. For every $s=1,\ldots,k,$ set $\mathbf{x}^{s}=\left(  x_{1}%
^{s},\ldots,x_{n}^{s}\right)  ,$ and note that $\sum_{i=1}^{n}x_{i}^{s}=0$
since $\mathbf{x}^{s}$ is orthogonal to the eigenspace of $q_{1}\left(
K_{n}\right)  ,$ which is $Span\left(  \mathbf{j}_{n}\right)  $.

To complete the proof of (A) we have to show that if $\overline{G}$ has
exactly $k$ bipartite components, then they are all balanced. For every
$s=1,\ldots,k,$ let $U_{s}$ and $W_{s}$ be the vertex classes of the $s$'th
bipartite component of $\overline{G}$ and write $V_{0\text{\ }}$ for the set
of vertices of $\overline{G}$ that do not belong to any bipartite component of
$\overline{G}.$ Since
\begin{equation}
0=q_{n+1-s}(\overline{G})=\left\langle Q(\overline{G})\mathbf{x}%
^{s},\mathbf{x}^{s}\right\rangle =\sum_{ij\in E\left(  \overline{G}\right)
}\left(  x_{i}^{s}+x_{j}^{s}\right)  ^{2}\label{Gb1}%
\end{equation}
for every edge $uv\in E\left(  \overline{G}\right)  ,$ we have $x_{u}%
^{s}=-x_{v}^{s}.$ It follows that $x_{u}^{s}=0$ if $u\in V_{0},$ and
$x_{i}^{s}=x_{j}^{s}$ if $i$ and $j$ belong to the same $U_{i}$ or the same
$W_{i}$. Hence for every $s=1,\ldots,k,$ there exist $k$ numbers $a_{1}%
^{s},\ldots,a_{k}^{s},$ such that
\[
x_{u}^{s}=\left\{
\begin{array}
[c]{rl}%
a_{i}^{s}, & \text{if }u\in U_{i};\\
-a_{i}^{s}, & \text{if }u\in W_{i};\\
0, & \text{if }u\in V_{0}.
\end{array}
\right.
\]
Now from $\sum_{i=1}^{n}x_{i}^{s}=0$ we see that $p_{1}a_{1}^{s}+\cdots
+p_{k}a_{k}^{s}=0,$ where $p_{i}=\left\vert W_{i}\right\vert -\left\vert
U_{i}\right\vert ,$ $1\leq i\leq k.$ Let $B$ be the $k\times n$ matrix whose
rows are the vectors $\mathbf{x}^{1},\ldots,\mathbf{x}^{k}.$ Since
$rank\left(  B\right)  =k,$ there exists $k$ independent columns of $B,$ say
the columns $\mathbf{c}_{1},\ldots,\mathbf{c}_{k}$ corresponding to the
vertices $u_{1},\ldots,u_{k}.$ Since every two of these columns are linearly
independent, the vertices $u_{1},\ldots,u_{k}$ belong to different components.
Define $q_{1},\ldots,q_{k}$ \ by%
\[
q_{i}=\left\{
\begin{array}
[c]{rl}%
p_{i}, & \text{if }u\in U_{i};\\
-p_{i}, & \text{if }u\in W_{i};
\end{array}
\right.
\]
and let%
\[
\mathbf{y}=\left(  y_{1},\ldots,y_{n}\right)  =q_{1}\mathbf{c}_{1}%
+\cdots+q_{k}\mathbf{c}_{k}.
\]
Our settings imply that for every $i=1,\ldots,k,$
\[
q_{i}\mathbf{c}_{i}=\left(  p_{i}a_{i}^{1},p_{i}a_{i}^{2},\ldots,p_{i}%
a_{i}^{k}\right)  ^{T}.
\]
Then, for every $s=1,\ldots,k,$
\[
y_{s}=p_{1}a_{1}^{s}+\cdots+p_{k}a_{k}^{s}=0,
\]
implying that $\mathbf{y}=0$ and so\textbf{ }$q_{1}=\cdots=q_{k}=0.$ Thus all
bipartite components of $\overline{G}$ are balanced, completing the proof of (A).

Now let us prove (B). Suppose that $\overline{G}$ has $k$ balanced bipartite
components, say with vertex classes $U_{i}$ and $W_{i},$ $i=1,\ldots,k.$ For
every $s=1,\ldots,k,$ define a vector $\mathbf{x}^{s}=\left(  x_{1}^{s}%
,\ldots,x_{n}^{s}\right)  $ by%
\[
x_{u}^{s}=\left\{
\begin{array}
[c]{rl}%
1, & \text{if }u\in U_{s};\\
-1, & \text{if }u\in W_{s};\\
0, & \text{if }u\in V\left(  G\right)  \backslash\left(  U_{s}\cup
W_{s}\right)  .
\end{array}
\right.
\]
Since $\left\vert U_{s}\right\vert =\left\vert W_{s}\right\vert ,$ we see that
$\left\langle Q\left(  \overline{G}\right)  \mathbf{x}^{s},\mathbf{x}%
^{s}\right\rangle =0;$ also, $\sum_{j=1}^{n}x_{j}^{s}=0$ and so $\mathbf{x}%
^{s}$ is an eigenvector to $q_{s+1}\left(  K_{n}\right)  .$ Therefore, $n-2$
is an eigenvalue to $Q(G)$ with eigenvector $\mathbf{x}^{s}.$ As the vectors
$\mathbf{x}^{1},\ldots,\mathbf{x}^{k}$ are orthogonal, we see that that $n-2$
is an eigenvalue of $Q\left(  G\right)  $ with multiplicity at least $k.$ To
complete the proof in this case, note that none of the vectors $\mathbf{x}%
^{1},\ldots,\mathbf{x}^{k}$ can be an eigenvector to $q_{1}\left(  G\right)  $
since $G$ is connected and each of the vectors $\mathbf{x}^{1},\ldots
,\mathbf{x}^{k}$ has negative entries.

Now assume that $\overline{G}$ has $k+1$ bipartite components, let $U_{s}$ and
$W_{s}$ be the vertex classes of the $s$'th bipartite component of
$\overline{G}$ and set $p_{s}=\left\vert W_{s}\right\vert -\left\vert
U_{s}\right\vert .$ Write $V_{0\text{\ }}$ for the set of all vertices of
$\overline{G}$ that do not belong to any bipartite component of $\overline
{G}.$

To complete the proof of (B), we shall show that there exists $k$ linearly
independent vectors $\mathbf{y}^{i},$ each orthogonal to $\mathbf{j}_{n}$ and
satisfying $\left\langle Q\left(  \overline{G}\right)  \mathbf{y}%
^{i},\mathbf{y}^{i}\right\rangle =0.$ Indeed, in this case each $\mathbf{y}%
^{i}$ is an eigenvector to $q_{2}\left(  K_{n}\right)  =n-2$ and to
$q_{n}\left(  \overline{G}\right)  =0,$ implying that $Q\left(  G\right)
\mathbf{y}^{i}=\left(  n-2\right)  \mathbf{y}^{i}.$ Hence $n-2$ is an
eigenvalue of $Q\left(  G\right)  $ with multiplicity at least $k.$

Consider the $k$-dimensional linear space $L$ of all $\left(  k+1\right)
$-vectors $\left(  a_{1},\ldots,a_{k+1}\right)  $ satisfying%
\[
p_{1}a_{1}+\cdots+p_{k+1}a_{k+1}=0
\]
and choose $k$ linearly independent vectors $\mathbf{a}^{1},\ldots
,\mathbf{a}^{k}$ in $L.$ Now, for every $s=1,\ldots,k,$ let $\mathbf{a}%
^{s}=\left(  a_{1}^{s},\ldots,a_{k+1}^{s}\right)  $ and define the $n$-vector
$\mathbf{y}^{s}=\left(  y_{1}^{s},\ldots,y_{n}^{s}\right)  $ by%
\[
y_{u}^{s}=\left\{
\begin{array}
[c]{rl}%
a_{i}^{s}, & \text{if }u\in U_{i};\\
-a_{i}^{s}, & \text{if }u\in W_{i};\\
0, & \text{if }u\in V_{0}.
\end{array}
\right.
\]
We shall show that the vectors $\mathbf{y}^{1},\ldots,\mathbf{y}^{k}$ satisfy
the requirements. Indeed for every $s=1,\ldots,k,$%
\[
\sum_{u\in V}y_{u}^{s}=p_{1}a_{1}+\cdots+p_{k+1}a_{k+1}=0;
\]
hence, each $\mathbf{y}^{s}$ is orthogonal to $\mathbf{j}_{n}.$ Also
$y_{u}^{s}=-y_{v}^{s}$ for every edge $uv\in E\left(  \overline{G}\right)  $;
hence $\left\langle Q\left(  \overline{G}\right)  \mathbf{y}^{s}%
,\mathbf{y}^{s}\right\rangle =0$ for every $s=1,\ldots,k.$

Finally assume that
\[
c_{1}\mathbf{y}^{1}+\cdots+c_{k}\mathbf{y}^{k}=0.
\]
For every $i=1,\ldots,k+1$, choose a vertex $u\in U_{i}$ and note that
\[
c_{1}a_{i}^{1}+\cdots+c_{k}a_{i}^{k}=c_{1}y_{u}^{1}+\cdots+c_{k}y_{u}%
^{k}=0,\text{ }%
\]
This implies that
\[
c_{1}\mathbf{a}^{1}+\cdots+c_{k}\mathbf{a}^{k}=0,
\]
and since $\mathbf{a}^{1},\ldots,\mathbf{a}^{k}$ are linearly independent, it
turns out that $c_{1}=$ $\cdots=c_{k}=0;$ hence $\mathbf{y}^{1},\ldots
,\mathbf{y}^{k}$ are also linearly independent. This completes the proof of
(B) and of Theorem \ref{th2}.
\end{proof}

\vspace{4mm}

\begin{proof}
[{\textbf{Proof of Theorem \ref{th3}}} ]Applying Weyl's inequality
(\ref{Wein1}) we find that
\begin{equation}
q_{2}(G)\geq\lambda_{2}(G)+\delta(G), \nonumber\label{eq:last}%
\end{equation}
where $\lambda_{2}(G)$ is the second largest eigenvalue of the adjacency
matrix of $G$. Since $\lambda_{2}(G)\geq0$ with equality only for the complete
multipartite graphs with possibly isolated vertices (this was proved by Smith
in \cite{Smi70}), it follows that $q_{2}(G)\geq\delta(G),$ and equality is
possible only if $G\ $is a complete multipartite graph. Let thus
$q_{2}(G)=\delta(G)$ and $G$ be a complete $r$-partite graph with part sizes
$n_{1}\leq n_{2}\leq\cdots\leq n_{r}.$ In this case, $q_{2}(G)=\delta
(G)=n-n_{r}.$ If $r=2,$ it is known, that
\[
q_{2}\left(  K_{n_{1},n_{2}}\right)  =\left\{
\begin{array}
[c]{rl}%
1, & \text{if }n_{1}=1,\\
n_{2}, & \text{if }n_{1}\geq2.
\end{array}
\right.
\]
Hence $n_{1}=1$ or $n_{1}=n_{2},$ which completes the proof for $r=2.$ Let now
$r\geq3.$ If $n_{1}\geq2,$ then $G$ contains $K_{n_{1},n-n_{1}},$ and so
\[
n-n_{r}=q_{2}(G)\geq n-n_{1},
\]
implying that $G$ is regular complete multipartite graph.

Let now $n_{1}=1.$ If $n_{2} \leq2,$ then $G$ contains $K_{2,n-2}$ and so
\[
n-n_{r}=q_{2}(G)=n-2.
\]
Therefore $G=K_{1,\ldots,1,2,\ldots,2}.$

If $n_{2}>2,$ then $G$ contains $K_{n_{2},n-n_{2}}$ and so
\[
n-n_{r}=q_{2}(G)\geq n-n_{2},
\]
implying that $G=K_{1,t,\ldots,t}$ for some $t\geq2.$ We shall show that this
is only possible if $t=2,$ or $r=3$ and $t=3.$

As proved in \cite{FAVJ10}, the characteristic polynomial of the $Q$-matrix of
the $\left(  r+1\right)  $-partite graph $G=K_{1,t,\ldots,t}$ satisfies
\[
P_{Q}(G,x)=(x-tr+t-1)^{r(t-1)}(x-tr+2t-1)^{(r-1)}(x^{2}-(2tr-2t+tr+1)x+2t^{2}%
r(r-1)),
\]
with roots
\[
tr-t+1 \text{ \ , \ } \frac{3tr-2t+1\pm\sqrt{t^{2}(r-2)^{2}+2t(3r-2)+1}}{2}
\text{ \ and \ } tr-2t+1.
\]
Since
\[
\frac{3tr-2t+1+\sqrt{t^{2}(r-2)^{2}+2t(3r-2)+1}}{2}>tr-t+1 > tr-2t+1,
\]
we see that
\[
q_{1}(G)=\frac{3tr-2t+1+\sqrt{t^{2}(r-2)^{2}+2t(3r-2)+1}}{2}.
\]

Also, if $t>2+\frac{1}{r-1},$ one can show that
\[
tr-t+1<\frac{3tr-2t+1-\sqrt{t^{2}(r-2)^{2}+2t(3r-2)+1}}{2},
\]
and so,
\[
q_{2}(G)=\frac{3tr-2t+1-\sqrt{t^{2}(r-2)^{2}+2t(3r-2)+1}}{2}>\delta(G).
\]

Finally, if $t\leq2+\frac{1}{r-1}$, we find that $q_{2}(G)=tr-t+1=\delta(G)$ ;
in that case, when $r=2$, the only feasible graphs are $K_{1,2,2}$ and
$K_{1,3,3}$; and when $r\geq3,$ the feasible graphs are of the type
$K_{1,2,\ldots,2},$ completing the proof.
\end{proof}

\end{document}